\documentclass[a4paper,english,fontsize=10pt,parskip=half,abstracton]{scrartcl}
\usepackage{babel}
\usepackage[utf8]{inputenc}
\usepackage[T1]{fontenc}
\usepackage[a4paper,left=20mm,right=20mm,top=30mm,bottom=30mm]{geometry}
\usepackage{amsmath}
\usepackage{amsthm}
\usepackage{amssymb}
\usepackage{enumerate}
\usepackage[bookmarks=false,
            pdftitle={Blocks with defect group D2n * C2m},
            pdfauthor={Benjamin Sambale},
            pdfkeywords={blocks, defect groups, dihedral groups, fusion systems, Alperin's conjecture},
            pdfstartview={FitH}]{hyperref}
\newtheorem{Theorem}{Theorem}[section]
\newtheorem{Lemma}[Theorem]{Lemma}

\renewcommand{\phi}{\varphi}
\newcommand{\C}{\operatorname{C}}

\newcommand{\N}{\operatorname{N}}
\newcommand{\Z}{\operatorname{Z}}
\newcommand{\cohom}{\operatorname{H}}
\newcommand{\Aut}{\operatorname{Aut}}

\newcommand{\Out}{\operatorname{Out}}
\newcommand{\pcore}{\operatorname{O}}

\newcommand{\SL}{\operatorname{SL}}
\newcommand{\Irr}{\operatorname{Irr}}
\newcommand{\IBr}{\operatorname{IBr}}

\newcommand{\Gal}{\operatorname{Gal}}

\newcommand{\Ob}{\operatorname{Ob}}
\newcommand{\Hom}{\operatorname{Hom}}

\mathchardef\ordinarycolon\mathcode`\:  
 \mathcode`\:=\string"8000
 \begingroup \catcode`\:=\active
   \gdef:{\mathrel{\mathop\ordinarycolon}}
 \endgroup

\title{Blocks with defect group $D_{2^n}\ast C_{2^m}$}
\author{
Benjamin Sambale\\
Mathematisches Institut\\
Friedrich-Schiller-Universität\\
07743 Jena\\
Germany\\
{\tt benjamin.sambale@uni-jena.de}
}
\date{\today}

\begin{document}
\frenchspacing
\maketitle
\begin{abstract}\noindent
We determine the numerical invariants of blocks with defect group $D_{2^n}\ast C_{2^m}\cong Q_{2^n}\ast C_{2^m}$ (central product), where $n\ge 3$ and $m\ge 2$. As a consequence, we prove Brauer's $k(B)$-conjecture, Olsson's conjecture (and more generally Eaton's conjecture), Brauer's height zero conjecture, the Alperin-McKay conjecture, Alperin's weight conjecture and Robinson's ordinary weight conjecture for these blocks. Moreover, we show that the gluing problem has a unique solution in this case. This paper continues \cite{SambaleD2nxC2m}.
\end{abstract}


\section{Introduction}
Let $R$ be a discrete complete valuation ring with quotient field $K$ of characteristic $0$. Moreover, let $(\pi)$ be the maximal ideal of $R$ and $F:=R/(\pi)$. We assume that $F$ is algebraically closed of characteristic $2$.
We fix a finite group $G$, and assume that $K$ contains all $|G|$-th roots of unity. Let $B$ be a $2$-block of $RG$ with defect group $D$. We denote the number of irreducible ordinary characters of $B$ by $k(B)$. These characters split in $k_i(B)$ characters of height $i\in\mathbb{N}_0$. Here the \emph{height} of a character $\chi$ in $B$ is the largest integer $h(\chi)\ge 0$ such that $2^{h(\chi)}|G:D|_2\mathrel{\big|}\chi(1)$, where $|G:D|_2$ denotes the highest $2$-power dividing $|G:D|$.
Finally, let $l(B)$ be the number of irreducible Brauer characters of $B$.

In \cite{SambaleD2nxC2m} we determined the invariants of $B$ in the case $D\cong D_{2^n}\times C_{2^m}$. In order to proceed with defect groups of the form $Q_{2^n}\times C_{2^m}$ it is necessary (for the induction step) to discuss central products  of the form $D_{2^n}\ast C_{2^m}$ first. Let
\[D:=\langle x,y,z\mid x^{2^{n-1}}=y^2=z^{2^m}=[x,z]=[y,z]=1,\ yxy^{-1}=x^{-1},\ x^{2^{n-2}}=z^{2^{m-1}}\rangle\cong D_{2^n}\ast C_{2^m},\]
where $n\ge 2$ and $m\ge1$. For $m=1$ we get $D\cong D_{2^n}$. Then the invariants of $B$ are known (see \cite{Brauer}). Hence, we assume $m\ge 2$. Similarly for $n=2$ we get $D=\langle y,z\rangle\cong C_2\times C_{2^m}$. Then $B$ is nilpotent and everything is known. Thus, we also assume $n\ge 3$. Then we have $D=\langle x,yz^{2^{m-2}},z\rangle\cong Q_{2^n}\ast C_{2^m}$.

The paper follows the lines of $\cite{SambaleD2nxC2m}$. However, the proof of the main theorem is a bit more complicated, since the upper bound for $k(B)$ in terms of Cartan invariants of major subsections is not sharp. Hence, it is necessary to consider generalized decomposition numbers and contributions. Here some of the calculations are similar to the quaternion case in \cite{Olsson}. Moreover, we introduce a new approach to construct a set of representatives for the conjugacy classes of subsections which uses only the fusion system of the block.

\section{Subsections}

The first lemma shows that the situation splits naturally in two cases according  to $n=3$ or $n\ge 4$.

\begin{Lemma}\label{aut}
The automorphism group $\Aut(D)$ is a $2$-group if and only if $n\ge 4$.
\end{Lemma}
\begin{proof}
Since $\Aut(Q_8)\cong S_4$, we see that $\Aut(Q_8\times C_{2^m})$ is not a $2$-group. An automorphism of $Q_8\times C_{2^m}$ of odd order acts trivially on $(Q_8\times C_{2^m})'\cong C_2$ and on $\Z(Q_8\times C_{2^m})/(Q_8\times C_{2^m})'\cong C_{2^m}$ and thus also on $\Z(Q_8\times C_{2^m})$ by Theorem~5.3.2 in \cite{Gorenstein}. Hence, $\Aut(Q_8\ast C_{2^m})=\Aut(D_8\ast C_{2^m})$ is not a $2$-group.

Now assume $n\ge 4$. Then $\Phi(D)=\langle x^2,z^2\rangle<\Phi(D)\Z(D)=\langle x^2,z\rangle$ are characteristic subgroups of $D$. Moreover, $\langle x,z\rangle$ is the only abelian maximal subgroup containing $\Phi(D)\Z(D)$. Hence, every automorphism of $\Aut(D)$ of odd order acts trivially on $D/\Phi(D)$. The claim follows from Theorem~5.1.4 in \cite{Gorenstein}.
\end{proof}

It follows that the inertial index $e(B)$ of $B$ equals $1$ for $n\ge 4$. In case $n=3$ there are two possibilities $e(B)\in\{1,3\}$, since $\Phi(D)\Z(D)$ is still characteristic in $D$. Now we investigate the fusion system $\mathcal{F}$ of the $B$-subpairs. For this we use the notation of \cite{Olssonsubpairs,Linckelmann}, and we assume that the reader is familiar with these articles. Let $b_D$ be a Brauer correspondent of $B$ in $RD\C_G(D)$. Then for every subgroup $Q\le D$ there is a unique block $b_Q$ of $RQ\C_G(Q)$ such that $(Q,b_Q)\le(D,b_D)$. We denote the inertial group of $b_Q$ in $\N_G(Q)$ by $\N_G(Q,b_Q)$. Then $\Aut_{\mathcal{F}}(Q)\cong\N_G(Q,b_Q)/\C_G(Q)$ and $\Out_{\mathcal{F}}(Q)\cong\N_G(Q,b_Q)/Q\C_G(Q)$.

\begin{Lemma}\label{essential}
Let $Q_1:=\langle x^{2^{n-3}},y,z\rangle\cong D_8\ast C_{2^m}$ and $Q_2:=\langle x^{2^{n-3}},xy,z\rangle\cong D_8\ast C_{2^m}$. 
Then $Q_1$ and $Q_2$ are the only candidates for proper $\mathcal{F}$-centric, $\mathcal{F}$-radical subgroups up to conjugation. In particular the fusion of subpairs is controlled by $\N_G(Q_1,b_{Q_1})\cup\N_G(Q_2,b_{Q_2})\cup D$.
Moreover, one of the following cases occurs:
\begin{enumerate}
\item[(aa)] $n=e(B)=3$ or ($n\ge 4$ and $\Out_{\mathcal{F}}(Q_1)\cong\Out_{\mathcal{F}}(Q_2)\cong S_3$).
\item[(ab)] $n\ge 4$, $\N_G(Q_1,b_{Q_1})=\N_D(Q_1)\C_G(Q_1)$, and $\Out_{\mathcal{F}}(Q_2)\cong S_3$.
\item[(ba)] $n\ge 4$, $\Out_{\mathcal{F}}(Q_1)\cong S_3$, and $\N_G(Q_2,b_{Q_2})=\N_D(Q_2)\C_G(Q_2)$.
\item[(bb)] $\N_G(Q_1,b_{Q_1})=\N_D(Q_1)\C_G(Q_1)$ and $\N_G(Q_2,b_{Q_2})=\N_D(Q_2)\C_G(Q_2)$.
\end{enumerate}
In case (bb) the block $B$ is nilpotent.
\end{Lemma}
\begin{proof}
Let $Q<D$ be $\mathcal{F}$-centric and $\mathcal{F}$-radical. Then $z\in\Z(D)\subseteq\C_D(Q)\subseteq Q$ and $Q=(Q\cap\langle x,y\rangle)\ast\langle z\rangle$. If $Q\cap\langle x,y\rangle$ is abelian, we have 
\begin{align*}
Q&=\langle x^iy,z\rangle\cong C_2\times C_{2^m}\hspace{5mm}\text{or}\\
Q&=\langle x,z\rangle\cong C_{2^n}\ast C_{2^m}\cong C_{2^{\max\{n,m\}}}\times C_{2^{\min\{n,m\}-1}}
\end{align*}
for some $i\in\mathbb{Z}$. In the first case $\Aut(Q)$ is a $2$-group, since $m\ge 2$. Then $\pcore_2(\Aut_{\mathcal{F}}(Q))\ne 1$. Thus, assume $Q=\langle x,z\rangle$.
The group $D\subseteq\N_G(Q,b_Q)$ acts trivially on $\Omega(Q)\subseteq\Z(D)$, while a nontrivial automorphism of $\Aut(Q)$ of odd order acts nontrivially on $\Omega(Q)$ (see Theorem~5.2.4 in \cite{Gorenstein}). This contradicts $\pcore_2(\Aut_{\mathcal{F}}(Q))=1$. (Moreover, by Lemma~5.4 in \cite{Linckelmann} we see that $\Aut_{\mathcal{F}}(Q)$ is a $2$-group.)

Hence by Lemma~\ref{aut}, $Q$ is isomorphic to $D_8\ast C_{2^m}$ and contains an element of the form $x^iy$. After conjugation with a suitable power of $x$ we may assume $Q\in\{Q_1,Q_2\}$. This shows the first claim.

The second claim follows from Alperin's fusion theorem. Here observe that in case $n=3$ we have $Q_1=Q_2=D$.

Let $S\le D$ be an arbitrary subgroup isomorphic to $D_8\ast C_{2^m}$. If $z\notin S$, then for $\langle S,z\rangle=(\langle S,z\rangle\cap\langle x,y\rangle)\langle z\rangle$ we have $\langle S,z\rangle'=S'\cong C_2$. However, this is impossible, since $\langle S,z\rangle\cap\langle x,y\rangle$ has at least order $16$.
This contradiction shows $z\in S$. Thus, $S$ is conjugate to $Q\in\{Q_1,Q_2\}$ under $D$. In particular $Q$ is fully $\mathcal{F}$-normalized (see Definition~2.2 in \cite{Linckelmann}). Hence, $\N_D(Q)\C_G(Q)/Q\C_G(Q)\cong\N_D(Q)/Q\cong C_2$ is a Sylow $2$-subgroup of $\Out_{\mathcal{F}}(Q)=\N_G(Q,b_Q)/Q\C_G(Q)$ by Proposition~2.5 in \cite{Linckelmann}. 
Assume $\N_D(Q)\C_G(Q)<\N_G(Q,b_Q)$. 
Since $\pcore_2(\Out_{\mathcal{F}}(Q))=1$ and $\lvert\Aut(Q)\rvert=2^k\cdot3$ for some $k\in\mathbb{N}$, we get $\Out_{\mathcal{F}}(Q)\cong S_3$.

The last claim follows from Alperin's fusion theorem and $e(B)=1$ (for $n\ge 4$).
\end{proof}

The naming of these cases is adopted from \cite{Brauer}. 
Since the cases (ab) and (ba) are symmetric, we ignore case (ba) for the rest of the paper. It is easy to see that $Q_1$ and $Q_2$ are not conjugate in $D$ if $n\ge 4$. Hence, by Alperin's fusion theorem the subpairs $(Q_1,b_{Q_1})$ and $(Q_2,b_{Q_2})$ are not conjugate in $G$. It is also easy to see that $Q_1$ and $Q_2$ are always $\mathcal{F}$-centric.

\begin{Lemma}\label{fixedpt}
Let $Q\in\{Q_1,Q_2\}$ such that $\N_G(Q,b_Q)/Q\C_G(Q)\cong S_3$. Then 
\[\C_Q(\N_G(Q,b_Q))=\Z(Q)=\langle x^{2^{n-2}},z\rangle.\]
\end{Lemma}
\begin{proof}
Since $Q\subseteq\N_D(Q,b_Q)$, we have $\C_Q(\N_G(Q,b_Q))\subseteq\C_Q(Q)=\Z(Q)$.
On the other hand $\N_D(Q)$ and every automorphism of $\Aut_{\mathcal{F}}(Q)$ of odd order act trivially on $\Z(Q)=\Z(D)=\langle z\rangle\cong C_{2^m}$. Hence, the claim follows.
\end{proof}

In order to determine a set of representatives for the conjugacy classes of $B$-subsections, we introduce a general result which does not depend on $B$, $D$, or the characteristic of $F$.

\begin{Lemma}\label{repgeneral}
Let $\mathcal{R}$ be a set of representatives for the $\mathcal{F}$-conjugacy classes of elements of $D$ such that $\langle \alpha\rangle$ is fully $\mathcal{F}$-normalized for $\alpha\in\mathcal{R}$ ($\mathcal{R}$ always exists). Then
\[\bigl\{(\alpha,b_\alpha):\alpha\in\mathcal{R}\bigr\}\]
is a set of representatives for the $G$-conjugacy classes of $B$-subsections, where $b_\alpha:=b_{\langle\alpha\rangle}$ has defect group $\C_D(\alpha)$.
\end{Lemma}
\begin{proof}
Let $(\alpha,b)$ be an arbitrary $B$-subsection. Then $(\langle\alpha\rangle,b)$ is a $B$-subpair which lies in some Sylow $B$-subpair. Since all Sylow $B$-subpairs are conjugate in $\mathcal{F}$, we may assume $(\langle\alpha\rangle,b)\le(D,b_D)$. This shows $b=b_\alpha$. By the definition of $\mathcal{R}$ there exists a morphism $f$ in $\mathcal{F}$ such that $\beta:=f(\alpha)\in\mathcal{R}$. If we compose $f$ with inclusion maps from the left and the right, we get $f:\langle\alpha\rangle\to D$. Then the definition of $\mathcal{F}$ implies $f(\alpha,b_\alpha)=(\beta,b_\beta)$.

It is also easy to see that we can always choose a representative $\alpha$ such that $\langle\alpha\rangle$ is fully $\mathcal{F}$-normalized.

Now suppose that $(\alpha,b_\alpha)$ and $(\beta,b_\beta)$ with $\alpha,\beta\in\mathcal{R}$ are conjugate by $g\in G$. Then (with a slight abuse of notation) we have $g\in\Hom_{\mathcal{F}}(\langle\alpha\rangle,\langle\beta\rangle)$. Hence, $\alpha=\beta$.

It remains to prove that $b_\alpha$ has defect group $\C_D(\alpha)$ for $\alpha\in\mathcal{R}$. By Proposition~2.5 in \cite{Linckelmann} $\langle\alpha\rangle$ is also fully $\mathcal{F}$-centralized. Hence, Theorem~2.4(ii) in \cite{Linckelmann2} implies the claim.
\end{proof}

\begin{Lemma}\label{subrep}
The set $\mathcal{R}$ in the previous lemma is given as follows:
\begin{enumerate}[(i)]
\item $x^iz^j$ ($i=0,1,\ldots,2^{n-2}$, $j=0,1,\ldots,2^{m-1}-1$) in case (aa).\label{aarep}
\item $x^iz^j$ and $yz^j$ ($i=0,1,\ldots,2^{n-2}$, $j=0,1,\ldots,2^{m-1}-1$) in case (ab).\label{abrep}
\end{enumerate}
\end{Lemma}
\begin{proof}
By Lemma~\ref{fixedpt} in any case the elements $x^iz^j$ ($i=0,1,\ldots,2^{n-2}$, $j=0,1,\ldots,2^{m-1}-1$) are pairwise non-conjugate in $\mathcal{F}$. Moreover, $\langle x,z\rangle\subseteq\C_G(x^iz^j)$ and $|D:\N_D(\langle x^iz^j\rangle)|\le 2$.
Suppose that $\langle x^iyz^j\rangle\unlhd D$ for some $i,j\in\mathbb{Z}$. Then we have $x^{i+2}yz^j=x(x^iyz^j)x^{-1}\in\langle x^iyz^j\rangle$ and the contradiction $x^2\in\langle x^iyz^j\rangle$. This shows that the subgroups $\langle x^iz^j\rangle$ are always fully $\mathcal{F}$-normalized.

Assume that case (aa) occurs. Then the elements of the form $x^{2i}yz^j$ ($i,j\in\mathbb{Z}$) are conjugate to elements of the form $x^{2i}z^j$ under $D\cup\N_G(Q_1,b_{Q_1})$. Similarly, the elements of the form $x^{2i+1}yz^j$ ($i,j\in\mathbb{Z}$) are conjugate to elements of the form $x^{2i}z^j$ under $D\cup\N_G(Q_2,b_{Q_2})$. The claim follows in this case.

In case (ab) the given elements are pairwise non-conjugate, since no conjugate of $yz^j$ lies in $Q_2$. As in case (aa) the elements of the form $x^{2i}yz^j$ ($i,j\in\mathbb{Z}$) are conjugate to elements of the form $yz^j$ under $D$ and the elements of the form $x^{2i+1}yz^j$ ($i,j\in\mathbb{Z}$) are conjugate to elements of the form $x^{2i}z^j$ under $D\cup\N_G(Q_2,b_{Q_2})$. Finally, the subgroups $\langle yz^j\rangle$ are fully $\mathcal{F}$-normalized, since $yz^j$ is not conjugate to an element in $Q_2$.
\end{proof}

\section{The numbers \texorpdfstring{$k(B)$}{k(B)}, \texorpdfstring{$k_i(B)$}{ki(B)} and \texorpdfstring{$l(B)$}{l(B)}}

Now we study the generalized decomposition numbers of $B$. If $l(b_u)=1$, then we denote the unique irreducible modular character of $b_u$ by $\phi_u$. In this case the generalized decomposition numbers $d^u_{\chi\phi_u}$ for $\chi\in\Irr(B)$ form a column $d(u)$. Let $2^k$ be the order of $u$, and let $\zeta:=\zeta_{2^k}$ be a primitive $2^k$-th root of unity. Then the entries of $d(u)$ lie in the ring of integers $\mathbb{Z}[\zeta]$. Hence, there exist integers $a_i^u:=(a_i^{u}(\chi))_{\chi\in\Irr(B)}\in\mathbb{Z}^{k(B)}$ such that
\[d_{\chi\phi_u}^u=\sum_{i=0}^{2^{k-1}-1}{a_i^{u}(\chi)\zeta^i}.\]
We extend this by $a_{i+2^{k-1}}^u:=-a_i^u$ for all $i\in\mathbb{Z}$.

Let $|G|=2^ar$ where $2\nmid r$. We may assume $\mathbb{Q}(\zeta_{|G|})\subseteq K$. Then $\mathbb{Q}(\zeta_{|G|})\mid\mathbb{Q}(\zeta_r)$ is a Galois extension, and we denote the corresponding Galois group by
$\mathcal{G}:=\Gal\bigl(\mathbb{Q}(\zeta_{|G|})\mid\mathbb{Q}(\zeta_r)\bigr)$. Restriction gives an isomorphism
$\mathcal{G}\cong\Gal\bigl(\mathbb{Q}(\zeta_{2^a})\mid\mathbb{Q}\bigr)$.
In particular $|\mathcal{G}|=2^{a-1}$.
For every $\gamma\in\mathcal{G}$ there is a number $\widetilde{\gamma}\in\mathbb{N}$ such that $\gcd(\widetilde{\gamma},|G|)=1$, $\widetilde{\gamma}\equiv 1\pmod{r}$, and $\gamma(\zeta_{|G|})=\zeta_{|G|}^{\widetilde{\gamma}}$ hold. Then $\mathcal{G}$ acts on the set of subsections by
$^\gamma(u,b):=(u^{\widetilde{\gamma}},b)$.
For every $\gamma\in\mathcal{G}$ we get
\begin{equation}\label{dgamma}
d(u^{\widetilde{\gamma}})=\sum_{s\in\mathcal{S}}{a_s^{u}\zeta_{2^k}^{s\widetilde{\gamma}}}
\end{equation}
for every system $\mathcal{S}$ of representatives of the cosets of $2^{k-1}\mathbb{Z}$ in $\mathbb{Z}$.
It follows that
\begin{equation}\label{aiuspur}
a_s^u=2^{1-a}\sum_{\gamma\in\mathcal{G}}{d\bigl(u^{\widetilde{\gamma}}\bigr)\zeta_{2^k}^{-\widetilde{\gamma}s}}
\end{equation}
for $s\in\mathcal{S}$. 

For sake of completeness, we state the following general lemma which does not depend on $D$.

\begin{Lemma}\label{heightzeroodd}
Let $(u,b_u)$ be a $B$-subsection with $|\langle u\rangle|=2^k$ and $l(b_u)=1$.
If $\chi\in\Irr(B)$ has height $0$, then the sum 
\begin{equation}\label{sum}
\sum_{i=0}^{2^{k-1}-1}{a_i^u(\chi)}
\end{equation}
is odd.
\end{Lemma}
\begin{proof}
See \cite{SambaleD2nxC2m}.
\end{proof}

As in \cite{SambaleD2nxC2m} we prove Olsson's conjecture first. 

\begin{Lemma}\label{olsson}
Olsson's conjecture $k_0(B)\le 2^{m+1}=|D:D'|$ is satisfied in all cases.
\end{Lemma}
\begin{proof}
Let $\gamma\in\mathcal{G}$ such that the restriction of $\gamma$ to $\mathbb{Q}(\zeta_{2^a})$ is the complex conjugation. Then $x^{\widetilde{\gamma}}=x^{-1}$. The block $b_x$ has defect group $\C_D(x)=\langle x,z\rangle$ by Lemma~\ref{repgeneral}. Since we have shown that $\Aut_{\mathcal{F}}(\langle x,z\rangle)$ is a $2$-group, $b_x$ is nilpotent. In particular $l(b_x)=1$.
Since the subsections $(x,b_x)$ and $(x^{-1},b_{x^{-1}})=(x^{-1},b_x)={^{\gamma}(x,b_x)}$ are conjugate by $y$, we have $d(x)=d(x^{\widetilde{\gamma}})$ and 
\begin{equation}\label{aiequal}
a_j^x(\chi)=a_{-j}^x(\chi)=-a_{2^{n-2}-j}^x(\chi)                                                                    \end{equation} 
for all $\chi\in\Irr(B)$ by Eq.~\eqref{dgamma}. In particular $a_{2^{n-3}}^x(\chi)=0$. By the orthogonality relations we have $(d(x),d(x))=|\langle x,z\rangle|=2^{n-2+m}$. On the other hand the subsections $(x,b_x)$ and $(x^i,b_{x^i})=(x^i,b_x)$ are not conjugate for odd $i\in\{3,5,\ldots,2^{n-2}-1\}$. Eq.~\eqref{aiuspur} implies 
\[(a_0^x,a_0^x)=2^{2(1-a)}\sum_{\gamma,\delta\in\mathcal{G}}{\bigl(d(x^{\widetilde{\gamma}}),d(x^{\widetilde{\delta}})\bigr)}=2^{2(1-a)}2^{2a-n+1}(d(x),d(x))=2^{m+1}.\]
Combining Eq.~\eqref{aiequal} with Lemma~\ref{heightzeroodd} we see that $a_0^x(\chi)\ne 0$ is odd for characters $\chi\in\Irr(B)$ of height $0$. This proves the lemma.
\end{proof}

We remark that Olsson's conjecture in case (bb) also follows from Lemma~\ref{essential}. Moreover, in case (ab) Olsson's conjecture follows easily from Theorem~3.1 in \cite{Robinson}.

\begin{Lemma}\label{valuation}
Let $\nu$ be the (exponential) valuation of $R$ and let $\zeta$ be a primitive $2^k$-th root of unity for $k\ge 2$. Then $0<\nu(1+\zeta)<1$.
\end{Lemma}
\begin{proof}
We prove this by induction on $k$. For $k=2$ we have $\zeta\in\{\pm i\}$, where $i=\sqrt{-1}$. Then $2\nu(1+i)=\nu((1+i)^2)=\nu(2i)=1$ and the claim follows. Now let $k\ge 3$. Then $2\nu(1+\zeta)=\nu((1+\zeta)^2)=\nu(1+\zeta^2+2\zeta)=\nu(1+\zeta^2)$, since $\nu(1+\zeta^2)<1=\nu(2\zeta)$ by induction.
\end{proof}

\begin{Theorem}\label{main}\hfill
\begin{enumerate}[(i)]
\item In case (aa) and $n=3$ we have $k(B)=2^{m-1}\cdot 7$, $k_0(B)=2^{m+1}$, $k_1(B)=2^{m-1}\cdot 3$, and $l(B)=3$.
\item In case (aa) and $n\ge 4$ we have $k(B)=2^{m-1}(2^{n-2}+5)$, $k_0(B)=2^{m+1}$, $k_1(B)=2^{m-1}(2^{n-2}-1)$, $k_{n-2}(B)=2^m$, and $l(B)=3$.
\item In case (ab) we have $k(B)=2^{m-1}(2^{n-2}+4)$, $k_0(B)=2^{m+1}$, $k_1(B)=2^{m-1}(2^{n-2}-1)$, $k_{n-2}(B)=2^{m-1}$, and $l(B)=2$.
\item In case (bb) we have $k(B)=2^{m-1}(2^{n-2}+3)$, $k_0(B)=2^{m+1}$, $k_1(B)=2^{m-1}(2^{n-2}-1)$, and $l(B)=1$.
\end{enumerate}
In particular Brauer's $k(B)$-conjecture, Brauer's height zero conjecture and the Alperin-McKay conjecture hold.
\end{Theorem}
\begin{proof}
Assume first that case (bb) occurs. Then $B$ is nilpotent and $k_i(B)$ is just the number $k_i(D)$ of irreducible characters of $D$ of degree $2^i$ ($i\ge0$) and $l(B)=1$. 
In particular $k_0(B)=|D:D'|=2^{m+1}$ and $k(B)=k(D)=2^{m-1}(2^{n-2}+3)$. Since $|D|$ is the sum of the squares of the degrees of the irreducible characters, we get $k_1(B)=k_1(D)=2^{m-1}(2^{n-2}-1)$.

Now assume that case (aa) or case (ab) occurs.
We determine the numbers $l(b)$ for the subsections in Lemma~\ref{subrep} and apply Theorem~5.9.4 in \cite{Nagao}. Let us begin with the nonmajor subsections. Since $\Aut_{\mathcal{F}}(\langle x,z\rangle)$ is a $2$-group, the block $b_{\langle x,z\rangle}$ with defect group $\langle x,z\rangle$ is nilpotent. Hence, we have $l(b_{x^iz^j})=1$ for all $i=1,\ldots,2^{n-2}-1$ and $j=0,1,\ldots,2^{m-1}-1$. The blocks $b_{yz^j}$ ($j=0,1,\ldots,2^{m-1}-1$) have $Q_1$ as defect group. Since $\N_G(Q_1,b_{Q_1})=\N_D(Q_1)\C_G(Q_1)$, they are also nilpotent, and it follows that $l(b_{yz^j})=1$. 

The major subsections of $B$ are given by $(z^j,b_{z^j})$ for $j=0,1,\ldots,2^m-1$ up to conjugation. By Lemma~\ref{fixedpt} the cases for $B$ and $b_{z^j}$ coincide. As usual, the blocks $b_{z^j}$ dominate blocks $\overline{b_{z^j}}$ of $R\C_G(z^j)/\langle z^j\rangle$ with defect group $D/\langle z^j\rangle\cong D_{2^{n-1}}\times C_{2^m/|\langle z^j\rangle|}$ for $j\ne 0$. By Theorem~5.8.11 in \cite{Nagao} we have $l(b_{z^j})=l(\overline{b_{z^j}})$. With the notations of \cite{SambaleD2nxC2m} the cases for $b_{z^j}$ and $\overline{b_{z^j}}$ also coincide (see Theorem~1.5 in \cite{Olsson}). Now we discuss the cases (ab) and (aa) separately.

\textbf{Case (ab):}\\
Then by Theorem~5.9.4 in \cite{Nagao} we have
\[k(B)-l(B)=2^{m-1}(2^{n-2}-1)+2^{m-1}+2(2^m-1)=2^{m-1}(2^{n-2}+4)-2.\]
Since $B$ is a centrally controlled block, we have $l(B)\ge l(b_{z})=2$ and
$k(B)\ge 2^{m-1}(2^{n-2}+4)$ (see Theorem~1.1 in \cite{KuelshammerOkuyama}). In order to bound $k(B)$ from above we study the numbers $d^z_{\chi\phi}$. Let $D^z:=(d^z_{\chi\phi_i})_{\substack{\chi\in\Irr(B),\\i=1,2}}$. Then $(D^z)^\text{T}\overline{D^z}=C^z$ is the Cartan matrix of $b_z$. Since $\overline{b_z}$ has defect group $D_{2^{n-1}}$, we get
\[C^z=2^m\begin{pmatrix}2^{n-3}+1&2\\2&4\end{pmatrix}\]
up to basic sets (see proof of Theorem~3.15 in \cite{Olsson}). Hence, Lemma~1 in \cite{SambalekB2} implies $k(B)\le 2^{m-1}(2^{n-2}+6)$. In order to derive a sharper bound, we consider the generalized decomposition numbers more carefully.
With a similar notation as above we write
\[d^z_{\chi\phi_i}=\sum_{j=0}^{2^{m-1}-1}{a^i_j(\chi)\zeta^j}\]
for $i=1,2$, where $\zeta$ is a primitive $2^m$-th root of unity. Since the subsections $(z^j,b_{z^j})$ are pairwise non-conjugate for $j=0,\ldots,2^m-1$, we get
\begin{align*}
(a^1_i,a^1_j)&=(2^{n-2}+2)\delta_{ij},&(a^1_i,a^2_j)&=4\delta_{ij},&(a^2_i,a^2_j)=8\delta_{ij}
\end{align*}
as in the proof of Lemma~\ref{olsson}. We introduce the matrix $M^z:=(m^z_{\chi\psi})_{\chi,\psi\in\Irr(B)}=2^{n+m-1}D^z(C^z)^{-1}\overline{D^z}^\text{T}$ of contributions. Then 
\[m^z_{\chi\psi}=4d^z_{\chi\phi_1}\overline{d^z_{\psi\phi_1}}-2(d^z_{\chi\phi_1}\overline{d^z_{\psi\phi_2}}+d^z_{\chi\phi_2}\overline{d^z_{\psi\phi_1}})+(2^{n-3}+1)d^z_{\chi\phi_2}\overline{d^z_{\psi\phi_2}}.\]
It follows from (5G) and (5H) in \cite{BrauerBlSec2} that
\begin{equation}\label{height0}
h(\chi)=0\Longleftrightarrow m^z_{\chi\chi}\in R^\times\Longleftrightarrow d^z_{\chi\phi_2}\in R^\times\Longleftrightarrow\sum_{j=0}^{2^{m-1}-1}a^2_j(\chi)\equiv 1\pmod{2}.
\end{equation}
Assume that $k(B)$ is as large as possible. Since $(z,b_z)$ is a major subsection, no row of $D^z$ vanishes. Hence, for $j\in\{0,1,\ldots,2^{m-1}-1\}$ we have essentially the following possibilities (where $\epsilon_1,\epsilon_2,\epsilon_3,\epsilon_4\in\{\pm1\}$; cf. proof of Theorem~3.15 in \cite{Olsson}):
\begin{align*}
(I)&:\left(\begin{array}{c|cccccccccccccc}
a^1_j&\pm1&\cdots&\pm1&\epsilon_1&\epsilon_2&\epsilon_3&\epsilon_4&.&\cdots&\cdots&\cdots&\cdots&\cdots&.\\[1mm]
a^2_j&.&\cdots&.&\epsilon_1&\epsilon_2&\epsilon_3&\epsilon_4&\pm1&\pm1&\pm1&\pm1&.&\cdots&.
\end{array}\right),\\
(II)&:\left(\begin{array}{c|ccccccccccc}
a^1_j&\pm1&\cdots&\pm1&\epsilon_1&\epsilon_2&\epsilon_3&.&\cdots&\cdots&\cdots&.\\[1mm]
a^2_j&.&\cdots&.&2\epsilon_1&\epsilon_2&\epsilon_3&\pm1&\pm1&.&\cdots&.
\end{array}\right),\\
(III)&:\left(\begin{array}{c|cccccccccccccc}
a^1_j&\pm1&\cdots&\pm1&\epsilon_1&\epsilon_2&.&\cdots&.\\[1mm]
a^2_j&.&\cdots&.&2\epsilon_1&2\epsilon_2&.&\cdots&.
\end{array}\right).
\end{align*}
The number $k(B)$ would be maximal if case (I) occurs for all $j$ and for every character $\chi\in\Irr(B)$ we have $\sum_{j=0}^{2^{m-1}-1}{|a^1_j(\chi)|}\le 1$ and $\sum_{j=0}^{2^{m-1}-1}{|a^2_j(\chi)|}\le 1$. However, this contradicts Lemma~\ref{olsson} and Equation~\eqref{height0}. This explains why we have to take the cases (II) and (III) also into account. Now let $\alpha$ (resp. $\gamma$, $\delta$) be the number of indices $j\in\{0,1,\ldots,2^{m-1}-1\}$ such that case (I) (resp. (II), (III)) occurs for $a^i_j$. Then obviously $\alpha+\beta+\gamma=2^{m-1}$. It is easy to see that we may assume for all $\chi\in\Irr(B)$ that $\sum_{j=0}^{2^{m-1}-1}{|a^1_j(\chi)|}\le 1$ in order to maximize $k(B)$. In contrast to that it does make sense to have $a^2_j(\chi)\ne 0\ne a^2_k(\chi)$ for some $j\ne k$ in order to satisfy Olsson's conjecture in view of Equation~\eqref{height0}. Let $\delta$ be the number of pairs $(\chi,j)\in\Irr(B)\times\{0,1,\ldots,2^{m-1}-1\}$ such that there exists a $k\ne j$ with $a^2_j(\chi)a^2_k(\chi)\ne 0$. Then it follows that 
\begin{align*}
\gamma&=2^{m-1}-\alpha-\beta,\\
k(B)&\le(2^{n-2}+6)\alpha+(2^{n-2}+4)\beta+(2^{n-2}+2)\gamma-\delta/2\\
&=2^{m+n-3}+6\alpha+4\beta+2\gamma-\delta/2\\
&=2^{m+n-3}+2^m+4\alpha+2\beta-\delta/2,\\
8\alpha+4\beta-\delta&\le k_0(B)\le 2^{m+1}.
\end{align*}
This gives $k(B)\le 2^{m+n-3}+2^{m+1}=2^{m-1}(2^{n-2}+4)$. Together with the lower bound above, we have shown that $k(B)=2^{m-1}(2^{n-2}+4)$ and $l(B)=2$. In particular the cases (I), (II) and (III) are really the only possibilities which can occur. The inequalities above imply also $k_0(B)=2^{m+1}$. However we do not know the precise values of $\alpha$, $\beta$, $\gamma$, and $\delta$. We will see in a moment that $\delta=0$. Assume the contrary. If $\chi\in\Irr(B)$ is a character such that $a^2_j(\chi)a^2_k(\chi)\ne 0$ for some $j\ne k$, then it is easy to see that $a^2_j(\chi)a^2_k(\chi)\in\{\pm1\}$ and $a^2_l(\chi)=0$ for all $l\notin\{j,k\}$. For if not, we would have $8\alpha+4\beta-\delta<k_0(B)$ or $k(B)<2^{m+n-3}+2^m+4\alpha+2\beta-\delta/2$. Hence, we have to exclude the following types of rows of $D^z$ (where $\epsilon\in\{\pm1\}$):
$(\epsilon\zeta^j,\epsilon\zeta^j+\epsilon\zeta^k)$, $(\epsilon\zeta^j,\epsilon\zeta^j-\epsilon\zeta^k)$, $(0,\epsilon\zeta^j+\epsilon\zeta^k)$, and $(0,\epsilon\zeta^j-\epsilon\zeta^k)$. 
Let $d^z_{\chi.}$ be the row of $D^z$ corresponding to the character $\chi\in\Irr(B)$.
If $d^z_{\chi.}=(\epsilon\zeta^j,\epsilon\zeta^j+\epsilon\zeta^k)$ for $j\ne k$ we have \[m^z_{\chi\chi}=4-2(2+\zeta^{j-k}+\zeta^{k-j})+(2^{n-3}+1)(2+\zeta^{j-k}+\zeta^{k-j})=4+(2^{n-3}-1)(2+\zeta^{j-k}+\zeta^{k-j}).\]
Since $\nu(\zeta^{j-k}+\zeta^{k-j})=\nu(\zeta^{j-k}(\zeta^{j-k}+\zeta^{k-j}))=\nu(1+\zeta^{2(j-k)})$, Lemma~\ref{valuation} implies $\nu(2+\zeta^{j-k}+\zeta^{k-j})\le 1$. This yields the contradiction $1\le h(\chi)<\nu(m^z_{\chi\chi})\le 1$. A very similar calculation works for the other types of rows. Thus, we have shown $\delta=0$. Then the rows of $D^z$ have the following forms: $(\pm\zeta^j,0)$, $(\epsilon\zeta^j,\epsilon\zeta^j)$, $(0,\pm\zeta^j)$, and $(\epsilon\zeta^j,2\epsilon\zeta^j)$. We already know which of these rows correspond to characters of height $0$. In order to determine $k_i(B)$ we calculate the contributions for the remaining rows. If $d^z_{\chi.}=(\pm\zeta^j,0)$, we have $m^z_{\chi\chi}=4$. Then (5G) in \cite{BrauerBlSec2} implies $h(\chi)=1$. 
The number of these rows is precisely \[(2^{n-2}-2)\alpha+(2^{n-2}-1)\beta+2^{n-2}\gamma=2^{n+m-3}-2\alpha-\beta=2^{n+m-3}-2^{m-1}=2^{m-1}(2^{n-2}-1).\]
Now assume that $\psi\in\Irr(B)$ is a character of height $0$ such that $d^z_{\psi.}=(0,\pm\zeta^j)$ (such characters always exist). Let $\chi\in\Irr(B)$ such that $d^z_{\chi.}=(\epsilon,2\epsilon)$, where $\epsilon\in\{\pm1\}$.
Then $m^z_{\chi\psi}=-2(\pm\epsilon\zeta^{k-j})+(2^{n-3}+1)(\pm\epsilon2\zeta^{k-j})=\pm\epsilon2^{n-2}\zeta^{k-j}$, and (5H) in \cite{BrauerBlSec2} implies $h(\chi)=n-2$. The number of these characters is precisely $k(B)-k_0(B)-2^{m-1}(2^{n-2}-1)=2^{m-1}$. This gives $k_i(B)$ for $i\in\mathbb{N}$ (recall that $n\ge 4$ in case (ab)).

\textbf{Case (aa):}\\
Here the arguments are similar, so that we will leave out some details.
We have
\[k(B)-l(B)=2^{m-1}(2^{n-2}-1)+3(2^m-1)=2^{m-1}(2^{n-2}+5)-3.\]
Again $B$ is centrally controlled, and $l(B)\ge 3$ and $k(B)\ge 2^{m-1}(2^{n-2}+5)$ follow from Theorem~1.1 in \cite{KuelshammerOkuyama}. The Cartan matrix $C^z$ of $b_z$ is given by
\[C^z=2^m\begin{pmatrix}
2^{n-3}+1&1&1\\1&2&0\\1&0&2
\end{pmatrix}\]
up to basic sets (see proof of Theorem~3.17 in \cite{Olsson} and observe that we can remove the negative sign there). Lemma~1 in \cite{SambalekB2} gives the weak bound $k(B)\le 2^{m-1}(2^{n-2}+6)$. We write $\IBr(b_z)=\{\phi_1,\phi_2,\phi_3\}$ and define the integral columns $a^i_j$ for $i=1,2,3$ and $j=0,1,\ldots,2^{m-1}-1$ as above. Then we can calculate the scalar products $(a^i_j,a^k_l)$. In particular the orthogonality relations imply that the columns $a^2_j$ and $a^3_j$ consist of four entries $\pm1$ and zeros elsewhere. The contributions are given by
\begin{align*}
m^z_{\chi\psi}&=4d^z_{\chi\phi_1}\overline{d^z_{\psi\phi_1}}-2\bigl(d^z_{\chi\phi_1}\overline{d^z_{\psi\phi_2}}+d^z_{\chi\phi_2}\overline{d^z_{\psi\phi_1}}+d^z_{\chi\phi_1}\overline{d^z_{\psi\phi_3}}+d^z_{\chi\phi_3}\overline{d^z_{\psi\phi_1}}\bigr)\\
&\mathrel{\phantom{=}}+d^z_{\chi\phi_2}\overline{d^z_{\psi\phi_3}}+d^z_{\chi\phi_3}\overline{d^z_{\psi\phi_2}}+(2^{n-2}+1)\bigl(d^z_{\chi\phi_2}\overline{d^z_{\psi\phi_2}}+d^z_{\chi\phi_3}\overline{d^z_{\psi\phi_3}}\bigr)
\end{align*}
for $\chi,\psi\in\Irr(B)$.
As before (5H) in \cite{BrauerBlSec2} implies
\begin{equation}\label{height0b}
\begin{split}
h(\chi)=0&\Longleftrightarrow m^z_{\chi\chi}\in R^\times\Longleftrightarrow|d^z_{\chi\phi_2}+d^z_{\chi\phi_3}|^2\in R^\times\\
&\Longleftrightarrow d^z_{\chi\phi_2}+d^z_{\chi\phi_3}\in R^\times\Longleftrightarrow\sum_{j=0}^{2^{m-1}-1}{\bigl(a^2_j(\chi)+a^3_j(\chi)\bigr)}\equiv 1\pmod{2}.
\end{split}
\end{equation}
In order to search the maximum value for $k(B)$ (in view of Lemma~\ref{olsson} and Equation~\eqref{height0b}) we have to consider the following possibilities (where $\epsilon_1,\epsilon_2,\epsilon_3,\epsilon_4\in\{\pm1\}$):
\begin{align*}
(I)&:\left(\begin{array}{c|cccccccccccccc}
a^1_j&\pm1&\cdots&\pm1&\epsilon_1&\epsilon_2&\epsilon_3&\epsilon_4&.&\cdots&\cdots&\cdots&\cdots&\cdots&.\\[1mm]
a^2_j&.&\cdots&.&\epsilon_1&\epsilon_2&.&.&\pm1&\pm1&.&\cdots&\cdots&\cdots&.\\[1mm]
a^3_j&.&\cdots&\cdots&\cdots&.&\epsilon_3&\epsilon_4&.&.&\pm1&\pm1&.&\cdots&.
\end{array}\right),\\
(II)&:\left(\begin{array}{c|cccccccccccc}
a^1_j&\pm1&\cdots&\pm1&\epsilon_1&\epsilon_2&\epsilon_3&.&\cdots&\cdots&\cdots&\cdots&.\\[1mm]
a^2_j&.&\cdots&.&\epsilon_1&\epsilon_2&.&\epsilon_4&\pm1&.&\cdots&\cdots&.\\[1mm]
a^3_j&.&\cdots&\cdots&.&\epsilon_2&\epsilon_3&-\epsilon_4&.&\pm1&.&\cdots&.
\end{array}\right),\\
(III)&:\left(\begin{array}{c|cccccccccc}
a^1_j&\pm1&\cdots&\pm1&\epsilon_1&\epsilon_2&.&\cdots&\cdots&\cdots&.\\[1mm]
a^2_j&.&\cdots&.&\epsilon_1&\epsilon_2&\epsilon_3&\epsilon_4&.&\cdots&.\\[1mm]
a^3_j&.&\cdots&.&\epsilon_1&\epsilon_2&-\epsilon_3&-\epsilon_4&.&\cdots&.
\end{array}\right).
\end{align*}
We define $\alpha$, $\beta$ and $\gamma$ as in case (ab).
Then we have $\alpha+\beta+\gamma=2^{m-1}$. Let $\delta$ be the number of triples $(\chi,i,j)\in\Irr(B)\times\{2,3\}\times\{0,1,\ldots,2^{m-1}-1\}$ such that there exists a $k\ne j$ with $a^i_j(\chi)a^2_k(\chi)\ne 0$ or $a^i_j(\chi)a^3_k(\chi)\ne 0$. Then the following holds:
\begin{align*}
\gamma&=2^{m-1}-\alpha-\beta,\\
k(B)&\le(2^{n-2}+6)\alpha+(2^{n-2}+5)\beta+(2^{n-2}+4)\gamma-\delta/2\\
&=2^{n+m-3}+2^{m+1}+2\alpha+\beta-\delta/2,\\
8\alpha+4\beta-\delta&\le k_0(B)\le2^{m+1}.
\end{align*}
This gives $k(B)\le 2^{n+m-3}+2^{m+1}+2^{m-1}=2^{m-1}(2^{n-2}+5)$. Together with the lower bound we have shown that $k(B)=2^{m-1}(2^{n-2}+5)$, $k_0(B)=2^{m+1}$, and $l(B)=3$. In particular the maximal value for $k(B)$ is indeed attended. Moreover, $\delta=0$. Let $\chi\in\Irr(B)$ such that $d^z_{\chi.}=(\pm\zeta^j,0,0)$. Then $m^z_{\chi\chi}=4$ and $h(\chi)=1$ by (5G) in \cite{BrauerBlSec2}. The number of these characters is 
\[(2^{n-2}-2)\alpha+(2^{n-2}-1)\beta+2^{n-2}\gamma=2^{n+m-1}-2^{m-1}=2^{m-1}(2^{n-2}-1).\]
Now let $\psi\in\Irr(B)$ a character of height $0$ such that $d^z_{\psi.}=(0,0,\pm\zeta^j)$, and let $\chi\in\Irr(B)$ such that $d^z_{\chi.}=(\epsilon\zeta^k,\epsilon\zeta^k,\epsilon\zeta^k)$, where $\epsilon\in\{\pm1\}$. Then we have $m^z_{\chi\psi}=-2(\pm\epsilon\zeta^{k-j})\pm\epsilon\zeta^{k-j}+(2^{n-2}+1)(\pm\epsilon\zeta^{k-j})=\pm\epsilon2^{n-2}\zeta^{k-j}$ and $h(\chi)=n-2$. The same holds if $d^z_{\chi.}=(0,\epsilon\zeta^k,-\epsilon\zeta^k)$. This gives the numbers $k_i(B)$ for $i\in\mathbb{N}$. Observe that we have to add $k_1(B)$ and $k_{n-2}(B)$ in case $n=3$.
\end{proof}

We add some remarks. It is easy to see that also Eaton's conjecture is satisfied which provides a generalization of Brauer's $k(B)$-conjecture and Olsson's conjecture (see \cite{Eaton}). Brauer's $k(B)$-conjecture already follows from Theorem~2 in \cite{SambalekB}. If we take $m=1$ in the formulas for $k_i(B)$ and $l(B)$ we get exactly the invariants for the defect group $Q_{2^n}$ (see \cite{Olsson}). However, recall that $D_{2^n}\ast C_2\cong D_{2^n}$. The principal block of $D$ gives an example for case (bb). For $n=3$ the principal block of $D\rtimes C_3$ gives an example for case (aa). If $n=4$, the principal blocks of $\SL(2,7)\ast C_{2^m}$ and $\texttt{SmallGroup(48,28)}\ast C_{2^m}$ show that also the cases (aa) and (ab) can occur (this can be seen with GAP).\nocite{GAP4}

\section{Alperin's weight conjecture}

In this section we will prove Alperin's weight conjecture using Proposition~5.4 in \cite{Kessar}.

\begin{Theorem}
Alperin's weight conjecture holds for $B$.
\end{Theorem}
\begin{proof}
Let $Q\le D$ be $\mathcal{F}$-centric and $\mathcal{F}$-radical. By Lemma~\ref{essential} we have $\Out_{\mathcal{F}}(Q)\cong S_3$, $\Out_{\mathcal{F}}(Q)\cong C_3$, or $\Out_{\mathcal{F}}(Q)=1$ (in the last two cases we have $Q=D$). In particular $\Out_{\mathcal{F}}(Q)$ has trivial Schur multiplier. Moreover, the group algebras $F1$ and $FS_3$ have precisely one block of defect $0$, while $FC_3$ has three blocks of defect $0$. Now the claim follows from Theorem~\ref{main} and Proposition~5.4 in \cite{Kessar}.
\end{proof}

\section{Ordinary weight conjecture}
In this section we prove Robinson's ordinary weight conjecture (OWC) for $B$ (see \cite{OWC}). If OWC holds for all groups and all blocks, then also Alperin's weight conjecture holds. However, for our particular block $B$ this implication is not known. In the same sense OWC is equivalent to Dade's projective conjecture (see \cite{Eaton}). 
For $\chi\in\Irr(B)$ let $d(\chi):=n+m-1-h(\chi)$ be the \emph{defect} of $\chi$. We set $k^i(B)=|\{\chi\in\Irr(B):d(\chi)=i\}|$ for $i\in\mathbb{N}$. 

\begin{Lemma}\label{chartable}
Let $\zeta$ be a primitive $2^m$-th root of unity. Then for $n=3$ the (ordinary) character table of $D$ is given as follows:
\[
\begin{tabular}{c|ccc}
1&$x$&$y$&$z$\\[0.6mm]\hline\\[-3.5mm]
1&1&1&$\zeta^{2r}$\\[1mm]
1&-1&1&$\zeta^{2r}$\\[1mm]
1&1&-1&$\zeta^{2r}$\\[1mm]
1&-1&-1&$\zeta^{2r}$\\[1mm]
2&0&$0$&$\zeta^{2r+1}$
\end{tabular}\]
where $r=0,1,\ldots,2^{m-1}-1$.
\end{Lemma}
\begin{proof}
We just take the characters $\chi\in\Irr(D_8\times C_{2^m})$ with $\chi(x^2z^{2^{m-1}})=\chi(1)$.
\end{proof}

\begin{Theorem}
The ordinary weight conjecture holds for $B$.
\end{Theorem}
\begin{proof}
We prove the version in Conjecture~6.5 in \cite{Kessar}. We may assume that $B$ is not nilpotent, and thus case (bb) does not occur. Suppose that $n=3$ and case (aa) occurs. Then $D$ is the only $\mathcal{F}$-centric, $\mathcal{F}$-radical subgroup of $D$. Since $\Out_{\mathcal{F}}(D)\cong C_3$, the set $\mathcal{N}_D$ consists only of the trivial chain (with the notations of \cite{Kessar}). We have $\textbf{w}(D,d)=0$ for $d\notin\{m+1,m+2\}$, since then $k^d(Q)=0$. For $d=m+1$ we get $\textbf{w}(D,d)=3\cdot 2^{m-1}$ by Lemma~\ref{chartable}. In case $d=m+2$ it follows that $\textbf{w}(D,d)=3\cdot 2^{m-1}+2^{m-1}=2^{m+1}$. Hence, OWC follows from Theorem~\ref{main}.

Now let $n\ge 4$ and assume that case (aa) occurs. Then there are three $\mathcal{F}$-centric, $\mathcal{F}$-radical subgroups up to conjugation: $Q_1$, $Q_2$ and $D$. Since $\Out_{\mathcal{F}}(D)=1$, it follows easily that $\textbf{w}(D,d)=k^d(D)$ for all $d\in\mathbb{N}$. By Theorem~\ref{main} it suffices to show
\[\textbf{w}(Q,d)=\begin{cases} 2^{m-1}&\text{if }d=m+1\\0&\text{otherwise}\end{cases}\]
for $Q\in\{Q_1,Q_2\}$, because $k^{m+1}(B)=k_{n-2}(B)=2^m$. We already have $\textbf{w}(Q,d)=0$ unless $d\in\{m+1,m+2\}$. W.\,l.\,o.\,g. let $Q=Q_1$.

Let $d=m+1$. Up to conjugation $\mathcal{N}_Q$ consists of the trivial chain $\sigma:1$ and the chain $\tau:1<C$, where $C\le\Out_{\mathcal{F}}(Q)$ has order $2$. We consider the chain $\sigma$ first. Here $I(\sigma)=\Out_{\mathcal{F}}(Q)\cong S_3$ acts trivially on the characters of $D$ or defect $m+1$ by Lemma~\ref{chartable}. This contributes $2^{m-1}$ to the alternating sum of $\textbf{w}(Q,d)$. Now consider the chain $\tau$. Here $I(\tau)=C$ and $z(FC)=0$ (notation from \cite{Kessar}). Hence, the contribution of $\tau$ vanishes and we get $\textbf{w}(Q,d)=2^{m-1}$ as desired.

Let $d=m+2$. Then we have $I(\sigma,\mu)\cong S_3$ for every character $\mu\in\Irr(Q)$ with $\mu(x^{2^{n-3}})=\mu(y)=1$. For the other characters of $Q$ with defect $d$ we have $I(\sigma,\mu)\cong C_2$. Hence, the chain $\sigma$ contributes $2^{m-1}$ to the alternating sum. There are $2^m$ characters $\mu\in\Irr(D)$ which are not fixed under $I(\tau)=C$. Hence, they split in $2^{m-1}$ orbits of length $2$. For these characters we have $I(\tau,\mu)=1$. For the other irreducible characters $\mu$ of $D$ of defect $d$ we have $I(\tau,\mu)=C$. Thus, the contribution of $\tau$ to the alternating sum is $-2^{m-1}$. This shows $\textbf{w}(Q,d)=0$.

In case (ab) we have only two $\mathcal{F}$-centric, $\mathcal{F}$-radical subgroups: $Q_2$ and $D$. Since $k_{n-2}(B)=2^{m-1}$ in this case, the calculations above imply the result. 
\end{proof}

\section{The gluing problem}
Finally we show that the gluing problem (see Conjecture~4.2 in \cite{gluingprob}) for the block $B$ has a unique solution. 
We will not recall the very technical statement of the gluing problem. Instead we refer to \cite{Parkgluing} for most of the  notations. Observe that the field $F$ is denoted by $k$ in \cite{Parkgluing}.

\begin{Theorem}
The gluing problem for $B$ has a unique solution.
\end{Theorem}
\begin{proof}
Assume first that $n\ge 4$.
Let $\sigma$ be a chain of $\mathcal{F}$-centric subgroups of $D$, and let $Q\le D$ be the largest subgroup occurring in $\sigma$. Then as in the proof of Lemma~\ref{essential} we have $Q=(Q\cap\langle x,y\rangle)\ast\langle z\rangle$. If $Q\cap\langle x,y\rangle$ is abelian or $Q=D$, then $\Aut_{\mathcal{F}}(Q)$ and $\Aut_{\mathcal{F}}(\sigma)$ are $2$-groups. In this case we get $\cohom^i(\Aut_{\mathcal{F}}(\sigma),F^\times)=0$ for $i=1,2$ (see proof of Corollary~2.2 in \cite{Parkgluing}). 
Now assume that $Q\in\{Q_1,Q_2\}$ and $\Aut_{\mathcal{F}}(Q)\cong S_4$. Then it is easy to see that $Q$ does not contain a proper $\mathcal{F}$-centric subgroup. Hence, $\sigma$ consists only of $Q$ and $\Aut_{\mathcal{F}}(\sigma)=\Aut_{\mathcal{F}}(Q)$. Thus, also in this case we get $\cohom^i(\Aut_{\mathcal{F}}(\sigma),F^\times)=0$ for $i=1,2$. It follows that $\mathcal{A}_{\mathcal{F}}^i=0$ and $\cohom^0([S(\mathcal{F}^c)],\mathcal{A}^2_{\mathcal{F}})=\cohom^1([S(\mathcal{F}^c)],\mathcal{A}^1_{\mathcal{F}})=0$. Hence, by Theorem~1.1 in \cite{Parkgluing} the gluing problem has only the trivial solution.

Now let $n=3$. Then we have $\cohom^i(\Aut_{\mathcal{F}}(\sigma),F^\times)=0$ for $i=1,2$ unless $\sigma=D$ and case (aa) occurs. In this case $\Aut_{\mathcal{F}}(\sigma)=\Aut_{\mathcal{F}}(D)\cong A_4$. Here $\cohom^2(\Aut_{\mathcal{F}}(\sigma),F^\times)=0$, but $\cohom^1(\Aut_{\mathcal{F}}(\sigma),F^\times)\cong\cohom^1(A_4,F^\times)\cong\cohom^1(C_3,F^\times)\cong C_3$. 
Hence, we have to consider the situation more closely. Up to conjugation there are three chains of $\mathcal{F}$-centric subgroups: $Q:=\langle x,z\rangle$, $D$, and $Q<D$. 
Since $[S(\mathcal{F}^c)]$ is partially ordered by taking subchains, one can view $[S(\mathcal{F}^c)]$ as a category, where the morphisms are given by the pairs of ordered chains. In our case $[S(\mathcal{F}^c)]$ has precisely five morphisms. With the notations of \cite{Webb} the functor $\mathcal{A}_{\mathcal{F}}^1$ is a \emph{representation} of $[S(\mathcal{F}^c)]$ over $\mathbb{Z}$. Hence, we can view $\mathcal{A}_{\mathcal{F}}^1$ as a module $\mathcal{M}$ over the incidence algebra of $[S(\mathcal{F}^c)]$. More precisely, we have
\[\mathcal{M}:=\bigoplus_{a\in\Ob[S(\mathcal{F}^c)]}{\mathcal{A}_{\mathcal{F}}^1(a)}=\mathcal{A}_{\mathcal{F}}^1(D)\cong C_3.\]
Now we can determine $\cohom^1([S(\mathcal{F}^c)],\mathcal{A}_{\mathcal{F}}^1)$ using Lemma~6.2(2) in \cite{Webb}. For this let $d:\Hom[S(\mathcal{F}^c)]\to\mathcal{M}$ a derivation. Then we have $d(\alpha)=0$ for all $\alpha\in\Hom[S(\mathcal{F}^c)]$ with $\alpha\ne(D,D)=:\alpha_1$. Moreover, \[d(\alpha_1)=d(\alpha_1\alpha_1)=(\mathcal{A}_{\mathcal{F}}^1(\alpha_1))(d(\alpha_1))+d(\alpha_1)=2d(\alpha_1)=0.\] 
Hence, $\cohom^1([S(\mathcal{F}^c)],\mathcal{A}^1_{\mathcal{F}})=0$.
\end{proof}


\section*{Acknowledgment}
This work was partly supported by the “Deutsche Forschungsgemeinschaft”.

\end{document}